\documentclass[11pt,leqno]{amsart}
\usepackage[dvipsnames]{xcolor}
\usepackage{amsmath}
\usepackage{amsfonts}
\usepackage{amssymb}
\usepackage{graphicx}
\usepackage{hyperref} 
\usepackage{mathrsfs} 
\usepackage{relsize} 
\usepackage{comment}
\usepackage{bm}

\theoremstyle{plain}
\newtheorem{theorem}{Theorem}

\newtheorem{corollary}[theorem]{Corollary}
\newtheorem{definition}[theorem]{Definition}
\newtheorem{lemma}[theorem]{Lemma}
\newtheorem*{problem*}{Problem}

\newtheorem{proposition}[theorem]{Proposition}

\theoremstyle{definition}

\newtheorem{remark}[theorem]{Remark}

\numberwithin{equation}{section}
\numberwithin{theorem}{section}

\newcommand{\rr}{\mathbb{R}}

\newcommand{\ric}{\operatorname{Ric}}

\newcommand{\loc}{_{loc}}
\newcommand{\lpl}[2]{L^{#1}_{loc}(#2)}
\newcommand{\ccomp}[2]{C^{#1}_{c}(#2)}

\newcommand{\dvol}{\ \textnormal{dv}}
\newcommand{\dint}[1]{\ \textnormal{d} #1}
\newcommand{\system}[2]{\left\{\begin{array}{#1} #2 \end{array} \right.}
\newcommand{\inn}{\textnormal{in}\ }

\newcommand{\e}{\epsilon}

\begin{document}

\begin{abstract}
On a complete Riemannian manifold $(M,g)$, we consider $L^{p}_{loc}$ distributional solutions of the the differential inequality $-\Delta u + \lambda u \geq 0$ with $\lambda >0$ a locally bounded function that may decay to $0$ at infinity. Under suitable growth conditions on the $L^{p}$ norm of $u$ over geodesic balls, we obtain that any such solution must be nonnegative. This is a kind of generalized $L^{p}$-preservation property that can be read as a Liouville type property for nonnegative subsolutiuons of the equation $\Delta u \geq \lambda u$. An application of the analytic results to  $L^{p}$ growth estimates of the extrinsic distance of  complete minimal submanifolds is also given.
\end{abstract}


\title[$L^p\loc$ positivity preservation and Liouville-type theorems]{$L^p\loc$ positivity preservation and\\ Liouville-type theorems}

\author{Andrea Bisterzo}
\address{Universit\`a degli Studi di Milano-Bicocca\\ Dipartimento di Matematica e Applicazioni \\ Via Cozzi 55, 20126 Milano - ITALY}
\email{a.bisterzo@campus.unimib.it}

\author{Alberto Farina}
\address{Universit\'{e} Picardie Jules Verne, LAMFA, UMR CNRS 7352, 33 rue St. Leu, 80039 Amiens - FRANCE}
\email{alberto.farina@u-picardie.fr}

\author{Stefano Pigola}
\address{Universit\`a degli Studi di Milano-Bicocca\\ Dipartimento di Matematica e Applicazioni \\ Via Cozzi 55, 20126 Milano - ITALY}
\email{stefano.pigola@unimib.it}

\maketitle


\section{Introduction}
In order to set our work, we first recall the notion of \textit{differential inequality in the sense of distributions}. Let $(M,g)$ be a Riemannian manifold and $\lambda$ a smooth function over $M$. Given $f\in \lpl{1}{M}$, we say that a function $u\in \lpl{1}{M}$ satisfies $-\Delta u + \lambda u \geq f$ (respectively $\leq f$) in the sense of distributions if
\begin{align*}
\int_M u (-\Delta \varphi +\lambda \varphi) \dvol \geq \int_M f \varphi \dvol \quad (\textnormal{resp.}\ \leq)
\end{align*}
for every $0\leq \varphi\in \ccomp{\infty}{M}$. Using an integration by parts, one can easily see that the notion of differential inequality in the sense of distributions is a generalization of the notion of weak differential inequality, which involves $W^{1,1}$ functions.

\begin{definition}[Positivity preserving property]\label{Def:PPP}
Given a Riemannian manifold $(M,g)$ and a family of function $\mathcal{S}\subseteq \lpl{1}{M}$, we say that $M$ has the $\mathcal{S}$ \textnormal{positivity preserving property} if any function $u\in \mathcal{S}$ that satisfies $-\Delta u + u \geq 0$ in the sense of distributions is nonnegative almost everywhere in $M$.
\end{definition}

Historically, the notion of positivity preserving property is motivated by the work of M. Braverman, O. Milatovic and M. Shubin, \cite{braverman2002essential}, where the authors conjectured that every complete Riemannian manifold is $L^2$ positivity preserving. In particular, this conjecture stimulated the study of the correlation between completeness and $L^p$ positivity preserving property for any $p\in [1,+\infty]$.

After some partial results involving constraints on the geometry of the manifold at hand, and covering all cases $p\in[1,+\infty]$, in \cite{pigola2021p} (see also \cite{pigola2023approximation} and \cite{guneysu2023new}) the authors proved that any Riemannian manifold $(M,g)$ is $L^p$ positivity preserving for every $p\in (1,+\infty)$, under the only assumption that $M$ is complete.

For what concerns the case $p=+\infty$, the recent work \cite{bisterzo2022infty} points out that the $L^\infty$ positivity preservation  is in a certain sense transversal to the notion of (geodesic) completeness. Indeed, the authors showed that a necessary and sufficient condition for a Riemannian manifold to satisfy the $L^\infty$ positivity preserving property is the stochastic completeness of the space.

Modeled on Definition \ref{Def:PPP}, in what follows we will consider a notion of positivity preserving property for slightly more general differential operators. In particular, we will deal with operators of the form $\mathcal{L}:=-\Delta+\lambda$, where $\lambda$ is a positive and locally bounded function. In this context, the present work generalizes the result of \cite{pigola2021p} and  \cite{pigola2023approximation} for complete Riemannian manifolds providing the $\mathcal{S}_p$ positivity preservation for any $p\in (1,+\infty)$, where $\mathcal{S}_p$ is the family of locally $p$-integrable functions satisfying a certain growth condition depending on the decay rate of the potential $\lambda$ at infinity. To follow, we obtain two results for the case $p=1$ when $\lambda$ is a positive constant, under the assumption that there exists a family of suitable (exhausting) cut-off functions whose laplacians have a ``good'' decay.

We stress that the results we obtained can be read as $L^p$ Liouville-type theorems when one deals with nonnegative solutions to $\Delta u \geq \lambda u$. In this direction we have a more direct comparison with the existing literature where, typically, one introduces a further pointwise control on the growth of the function and requires much more regularity on the solution. In the next sections  we shall comment on these aspects.\smallskip

The paper is organized as follows. In Section \ref{Sec:PreliminaryResults} we  prove an integral inequality in low regularity which represents the core of the $L^{p}_{loc}$ argument, $1<p<+\infty$. Section \ref{Sec:Lploc} is devoted to a generalized $L^{p}_{loc}$ Positivity Preservation for distributional solutions of $-\Delta u + \lambda u \geq 0$ with a possibly decaying functions $\lambda(x)$. The case $p=1$ will be dealt with in Section \ref{Sec:L1loc} under additional curvature restrictions that guarantee the existence of the so called Laplacian cut-offs. In the final Section \ref{Sec:Applications} we present an application to complete Euclidean minimal submanifolds with an extrinsic distance growth measured in integral sense. This generalizes the well known fact that complete minimal submanifolds in Euclidean space  with quadratic-exponential volume growth must be unbounded; see \cite{Karp-MathAnn} and \cite{PRS-PAMS}.


\section{Some preliminary results}\label{Sec:PreliminaryResults}

In what follows, if $u$ is a real-valued function we denote
\begin{align*}
u^+:=\max\{u,0\} \quad \textnormal{and} \quad u^-:=\max\{-u,0\}.
\end{align*}
We start recalling the Brezis-Kato inequality in a general Riemannian setting. This result is obtained in \cite{pigola2023approximation} for the general inequality $\Delta u\geq f\in L^1\loc$.

\begin{proposition}[Brezis-Kato inequality]
Let $(M,g)$ be a possibly incomplete Riemannian manifold and $\lambda$ a measurable function.

If $u\in \lpl{1}{M}$ is so that $\lambda u \in \lpl{1}{M}$ and satisfies $-\Delta u + \lambda u \leq 0$ in the sense of distributions, then $-\Delta u^+ + \lambda u^+ \leq 0$ in the sense of distributions.
\end{proposition}

As a consequence, in the next proposition we get a refinement of the regularity result obtained in \cite{pigola2023approximation} for complete manifolds. The inequality \eqref{Eq1:Regularity} will be the key tool in the proof of the positivity preserving properties stated in Section \ref{Sec:Lploc}.

\begin{proposition}\label{Prop1:Regularity}
Let $(M,g)$ be a complete Riemannian manifold and $0\leq \lambda\in L^\infty\loc(M)$. Assume that $u\in \lpl{1}{M}$ satisfies $-\Delta u+\lambda u \geq 0$ in the sense of distributions.

Then, $u^-\in \lpl{\infty}{M}$ and $(u^-)^\frac{p}{2}\in W^{1,2}\loc(M)$ for every $p\in(1,+\infty)$. Moreover, $u^-$ satisfies
\begin{align}\label{Eq1:Regularity}
(p-1) \int_M \lambda (u^-)^p \varphi^2 \dvol \leq \int_M (u^-)^p |\nabla \varphi|^2 \dvol
\end{align}
for every $0\leq \varphi \in \ccomp{0,1}{M}$.
\end{proposition}
\begin{proof}
By the Brezis-Kato inequality, the function $u^- \in \lpl{1}{M}$ satisfies $\Delta u^-\geq \lambda u^-$ in the sense of distributions. Therefore, by \cite[Theorem 3.1]{pigola2021p} it follows that $u^- \in \lpl{\infty}{M}$ and $(u^-)^\frac{p}{2}\in W^{1,2}\loc(M)$ for every $p\in(1,+\infty)$.

To prove \eqref{Eq1:Regularity}, let $\delta>0$ and set $v_\delta:=u^-+\delta \in \lpl{\infty}{M}\cap W^{1,2}\loc(M)$. Clearly, for every $q>0$ the function $v_\delta^q$ belongs to $\lpl{\infty}{M}\cap W^{1,2}\loc(M)$ and by \cite[Lemma 5.4]{pigola2021p} its weak gradient satisfies
\begin{align}\label{App1:Regularity}
\nabla v^q_\delta = q v^{q-1}_\delta \nabla v_\delta.
\end{align}
Moreover, $\Delta v_\delta\geq \lambda u^-$ in the sense of distributions, implying
\begin{align*}
\int_M \lambda u^- \psi \dvol + \int_M g(\nabla v_\delta, \nabla \psi) \leq 0
\end{align*}
for every $\psi \in W^{1,2}\loc(M)$ that is nonnegative almost everywhere. In particular, choosing $\psi=v_\delta^{p-1}\varphi^2$ with $\varphi\in \ccomp{0,1}{M}$ and using \eqref{App1:Regularity}, we get
\begin{align*}
0\geq & \int_M \lambda u^- v_\delta^{p-1} \varphi^2 \dvol + (p-1)\int_M v_\delta^{p-2} \varphi^2 |\nabla v_\delta|^2 \dvol \\
&+ 2 \int_M \varphi v_\delta^{p-1} g(\nabla v_\delta, \nabla \varphi) \dvol.
\end{align*}
By Cauchy-Schwarz inequality and Young's inequality, for any $\e \in (0,p-1)$ we have
\begin{align*}
2\varphi v_\delta^{p-1} g(\nabla v_\delta,\nabla \varphi) & \geq - 2 \varphi v_\delta^{p-1}|\nabla v_\delta| |\nabla \varphi|\\
&\geq -\e \varphi^2 v_\delta^{p-2}|\nabla v_\delta|^2-\e^{-1} v_\delta^p |\nabla \varphi|^2
\end{align*}
and thus
\begin{align*}
0\geq & \int_M \lambda u^- v_\delta^{p-1} \varphi^2 \dvol + (p-1-\e)\int_M v_\delta^{p-2} \varphi^2 |\nabla v_\delta|^2 \dvol \\
&- \e^{-1} \int_M v_\delta^{p} |\nabla \varphi|^2 \dvol.
\end{align*}
As $\e\to p-1$ we get
\begin{align*}
(p-1) \int_M \lambda  u^- v_\delta^{p-1} \varphi^2 \dvol \leq \int_M v_\delta^p |\nabla \varphi|^2 \dvol
\end{align*}
that, together with the fact that
\begin{align*}
\begin{array}{rll}
\lambda u^- v_\delta^{p-1}&\xrightarrow[]{\delta \to 0} \lambda (u^-)^p & \inn \lpl{1}{M}\\
v_\delta^p &\xrightarrow[]{\delta \to 0} (u^-)^p & \inn \lpl{1}{M}
\end{array}
\end{align*}
by Dominated Convergence Theorem, implies
\begin{align*}
(p-1) \int_M \lambda (u^-)^p \varphi^2 \dvol \leq \int_M (u^-)^p |\nabla \varphi|^2 \dvol
\end{align*}
obtaining the claim.
\end{proof}


\section{$L^p\loc$ positivity preserving property}
\label{Sec:Lploc}

In this section we face up the question of the $L^p\loc$ positivity preserving property for $p\in (1,+\infty)$, considering complete Riemannian manifolds and not requiring any curvature assumption.

Clearly, if the manifold is non-compact, we do not have any control on the growth at ``infinity'' of (the $p$-norm of) the general function $u\in \lpl{p}{M}$, making it impossible to retrace step by step what has been done in \cite{pigola2021p} and \cite{pigola2023approximation} in the $L^p$ case.

In addition, we also point out that we cannot expect to obtain a genuine positivity preserving property on the whole family of functions $\lpl{p}{M}$. Indeed, if $\lambda$ is a positive constant, then $u(x)=-e^{\sqrt{\lambda} x }$ is a negative function that solves $-u''+\lambda u=0$ in $\rr$. So the $L^p\loc$ positivity preserving property fails in general complete Riemannian manifolds.

Taking into account what we have observed so far, it seems natural to limit ourselves to the class of $L^p\loc$ functions whose $p$-norms satisfy a suitable (sub-exponential) growth condition.

We start with the following iterative lemma.

\begin{lemma}\label{Lem4:Inequality4}
Let $A>0$ and $f:[A,+\infty)\to (0,+\infty)$ be a nondecreasing function. Suppose there exist $\alpha>0, \delta\geq 0, \beta \geq 1$ and $\gamma>0$ so that
\begin{align}\label{Eq1:Inequality4}
f(r)\leq \frac{1}{\alpha (1+r)^{-\delta} h^\gamma+\beta} f(r+h)
\end{align}
for every $r\geq A$ and every $h>0$.

Then, for every $h>0$ the function $f$ satisfies
\begin{align*}
f(R)\geq f(A) \left(\alpha (1+R-h)^{-\delta}h^\gamma+\beta\right)^{\frac{R-A}{h}-1}
\end{align*}
for every $R\geq A+h$.
\end{lemma}
\begin{proof}
Fixed $h>0$, by assumption we have $f(r)\leq \left(\alpha (1+r)^{-\delta} h^\gamma+\beta\right)^{-1} f(r+h)$ for any $r\geq A$. Iterating, for every $n\in \mathbb{N}$ we get
\begin{align*}
f(r) & \leq \left(\alpha (1+r)^{-\delta} h^\gamma+\beta\right)^{-1} f(r+h)\\
&\leq \left(\alpha (1+r)^{-\delta} h^\gamma+\beta\right)^{-1} \left(\alpha (1+r+h)^{-\delta} h^\gamma+\beta\right)^{-1} f(r+2h)\\
&\leq \left(\alpha (1+r+h)^{-\delta} h^\gamma+\beta\right)^{-2} f(r+2h)\\
&\leq ... \leq \left(\alpha (1+r+ (n-1) h)^{-\delta}h^\gamma+\beta\right)^{-n} f(r+nh)
\end{align*}
for any $r\geq A$. It follows that for every $R>A$
\begin{align*}
f(R)&\geq f(A+nh)\\
&\geq \left(\alpha (1+A+ (n-1) h)^{-\delta}h^\gamma+\beta\right)^n f(A)\\
&\geq \left(\alpha (1+A+ (n-1) h)^{-\delta}h^\gamma+\beta\right)^{\frac{R-A}{h}-1} f(A),
\end{align*}
where $n=n(R,A,h)$ is the unique natural number satisfying $A+(n+1)h\geq R\geq A+nh$. In particular, if $R\geq A+h$, then $\frac{R-A}{h}\geq 1$ obtaining
\begin{align*}
f(R)&\geq \left(\alpha (1+A+ (n-1) h)^{-\delta}h^\gamma+\beta\right)^{\frac{R-A}{h}-1} f(A)\\
&\geq \left(\alpha (1+R-h)^{-\delta}h^\gamma+\beta\right)^{\frac{R-A}{h}-1} f(A)
\end{align*}
since $\frac{R-A}{h}-1\geq n-1$. This concludes the proof.
\end{proof}

Combining Lemma \ref{Lem4:Inequality4} with Proposition \ref{Prop1:Regularity} and with the choice standard family of rotationally symmetric cut-off functions, we get the following theorem.

\begin{theorem}[Generalized $L^p\loc$ positivity preserving property]\label{Thm1:LpPPP} Let $(M,g)$ be a complete Riemannian manifold, $\lambda \in \lpl{\infty}{M}$ a positive function and $p\in (1,+\infty)$. Moreover, assume there exist $o\in M$ and a constant $C>0$ so that
\begin{align*}
\lambda(x)\geq \frac{C}{(1+d^M(x,o))^{2-\e}} \quad \forall x\in M,
\end{align*}
where $\e\in (0,2]$ and $d^M$ is the intrinsic distance on $M$.

If $u\in \lpl{p}{M}$ satisfies $-\Delta u+\lambda u \geq 0$ in the sense of distributions and
\begin{align}\label{Eq1:LpPPP}
\int_{B_R(o)} (u^-)^p \dvol = o\left(e^{\theta R^{\frac{\e}{2}}}\right) \quad as\ R\to +\infty,
\end{align}
where $\theta=\sqrt{\frac{(p-1)C}{e-1}}$, then $u\geq 0$.
\end{theorem}

\begin{remark}[A Liouville-type theorem]
It clearly follows that the unique nonpositive $L^p\loc$ distributional solution to $-\Delta u +\lambda u\geq 0$ that satisfies condition \eqref{Eq1:LpPPP} is the null function. In this sense, Theorem \ref{Thm1:LpPPP} can be read as an $L^p$ Liouville-type theorem.
\end{remark}

\begin{remark}
The case $\e>2$ can be considered by reducing the problem to the case $\e=2$, since
\begin{align*}
\lambda(x)\geq C(1+d^M(x,o))^{\e-2}\geq C \quad \forall x\in M.
\end{align*}
\end{remark}

\begin{proof}
Let $u\in \lpl{p}{M}$ be a distributional solution to $-\Delta u + \lambda u \geq 0$ satisfying \eqref{Eq1:LpPPP}. For any fixed $a>0$ and $b>a$, consider the function $\eta_{a,b} \in C^{0,1}([0,+\infty))$ so that
\begin{align*}
\system{ll}{\eta_{a,b} \equiv 1 & \inn [0,a]\\  \eta_{a,b}(t)=\frac{b-t}{b-a} & \inn [a,b]\\ \eta_{a,b} \equiv 0 & \inn [b,+\infty).}
\end{align*}
In particular, $|\eta_{a,b}'(t)|\leq \frac{1}{b-a}$ almost everywhere in $[0,+\infty)$.

Set $\varphi_{a,b}(x):= \eta_{a,b}(d(x,o))$, where $d(\cdot,\cdot)$ is the intrinsic distance on $M$. Then, $\varphi_{a,b}\in \ccomp{0,1}{M}$ and satisfies
\begin{align*}
\system{ll}{\varphi_{a,b} \geq 0 & \inn M\\ |\nabla \varphi_{a,b}(x)| \leq \frac{1}{b-a} & \textnormal{a.e.}\ \inn M \\ \varphi_{a,b}\equiv 0 & \inn M\setminus \overline{B_b(o)} \\ \varphi_{a,b}\equiv 1 & \inn B_a(o).}
\end{align*}
Using $\varphi=\varphi_{a,b}$ in \eqref{Eq1:Regularity}, we get
\begin{align*}
\frac{1}{(b-a)^2} \int_{B_b(o)\setminus B_a(o)} (u^-)^p \dvol &\geq (p-1)\int_{B_a(o)} \lambda(u^-)^p \dvol\\
&\geq (p-1) \int_{B_a(o)} \frac{C}{(1+d^M(\cdot,o))^{2-\e}} (u^-)^p \dvol\\
&\geq (p-1) \frac{C}{(1+a)^{2-\e}} \int_{B_a(o)} (u^-)^p \dvol
\end{align*}
and, by summing
\begin{align*}
\frac{1}{(b-a)^2}\int_{B_a(o)} (u^-)^p \dvol
\end{align*}
to both sides of previous inequality, we obtain
\begin{align*}
\left((p-1)  \frac{C}{(1+a)^{2-\e}} + \frac{1}{(b-a)^2} \right)& \int_{B_a(o)} (u^-)^p \dvol \leq  \frac{1}{(b-a)^2} \int_{B_b(o)} (u^-)^p \dvol
\end{align*}
for every fixed $a>0$ and $b>a$. In particular, it implies that
\begin{align}\label{App1:LpPPP}
\begin{split}
\int_{B_a(o)} (u^-)^p \dvol & \leq \frac{1}{ (p-1) C(1+a)^{\e-2} h^2 + 1} \int_{B_{a+h}(o)} (u^-)^p \dvol
\end{split}
\end{align}
for every $a>0$ and $h> 0$.

If we suppose that $u^- \neq 0$, then there exists $A>0$ so that
\begin{align*}
\int_{B_A(o)} (u^-)^p \dvol >0.
\end{align*}
By \eqref{App1:LpPPP} we can apply Lemma \ref{Lem4:Inequality4} to
\begin{align*}
f:a\mapsto \int_{B_a(o)} (u^-)^p \dvol
\end{align*}
in $[A,+\infty)$, with $\gamma=\e$, $\delta=2-\e$, $\alpha=(p-1)C$ and $\beta=1$ and we get that for any $h> 0$ and for any $R>A+h$ the function $f$ satisfies
\begin{align*}
f(R)\geq f(A)\Big((p-1) C  (1+R-h)^{\e-2} h^2 +1\Big)^{\frac{R-A}{h}-1}.
\end{align*}
If $0<\e<2$ we can take $h= R^{1-\frac{\e}{2}}\sqrt{\frac{e-1}{(p-1)C}}$, obtaining
\begin{align*}
f(R)&\geq f(A) \left((p-1)C  \frac{h^2}{(1+R-h)^{2-\e}}+1 \right)^{\frac{R-A}{h}-1}\\
& \geq f(A) \left((p-1)C  \frac{h^2}{(h+R-h)^{2-\e}}+1 \right)^{\frac{R-A}{h}-1}\\
& = f(A)  \left((p-1)C  \frac{h^2}{R^{2-\e}}+1 \right)^{\frac{R-A}{h}-1}\\
& = f(A) e^{-\frac{A}{h}-1} e^{\frac{R}{h}}\\
& \geq  \frac{f(A)e^{-1}}{2} e^{\theta R^{\frac{\e}{2}}}
\end{align*}
for every $R$ big enough so that
\begin{align*}
R>A+h, \quad h\geq 1 \quad \textnormal{and} \quad e^{-\frac{A}{h}}\geq \frac{1}{2}.
\end{align*}
Similarly, if $\e=2$ we can choose $h=\sqrt{\frac{e-1}{(p-1)C}}$, in order to get
\begin{align*}
f(R)&\geq f(A) \left((p-1) C h^2 + 1 \right)^{\frac{R-A}{h}-1}\\
&= f(A) e^{-\theta A-1} e^{\theta R}.
\end{align*}
In both cases we obtain a contradiction to \eqref{Eq1:LpPPP}, implying that $u^-=0$ almost everywhere, i.e. the claim.
\end{proof}

\begin{remark}
In the paper \cite{MARI2010665} by L. Mari, M. Rigoli and A.G. Setti, using the viewpoint of maximum principles at infinity for the $\varphi$-Laplacian, the authors proved a general \textit{a priori} estimate that, in our setting, reduces as follow.

\begin{theorem}[{{\cite[Theorem B]{MARI2010665}}}]\label{thm:MRS}
Let $(M,g)$ be a complete Riemannian manifold and $\lambda\in C(M)$ be a positive function satisfying
\begin{align*}
\lambda(x) \geq \frac{B}{r(x)^{2-\e}} \quad in\ M\setminus B_{R_0}(o)
\end{align*}
for some $\e \in (0,+\infty), B>0, R_0>0$ and $o\in M$.

Let $\sigma \geq 0$ and $u\in C^1(M)$ be a distributional solution to
\begin{align*}
-\Delta u + \lambda u \geq 0 \quad in\ M
\end{align*}
so that either $u^-(x)=o(r(x)^\sigma)$ as $r(x)\to +\infty$, if $\sigma>0$, or $u$ is bounded from below, if $\sigma=0$. Lastly, assume
\begin{align*}
\liminf_{r\to +\infty} \frac{\ln |B_r(o)|}{r^{\e-\sigma}}<+\infty \quad if\ \sigma<\e
\end{align*}
or
\begin{align*}
\liminf_{r\to + \infty} \frac{\ln |B_r(o)|}{\ln r}<+\infty \quad if\ \sigma=\e.
\end{align*}
Then, $u\geq 0$.
\end{theorem}
This result compares with our Theorem \ref{Thm1:LpPPP}. Indeed, on the one hand, if we assume the pointwise control $u^-(x)=o(r^\sigma(x))$, for $0<\sigma<\e$, condition \eqref{Eq1:LpPPP} is satisfied provided $|B_R|=O(R^{-p\sigma}e^{\theta R^\frac{\e}{2}})$, $p \in (1,+\infty)$, while Theorem \ref{thm:MRS} requires the volume growth $|B_R|=O(e^{R^{\e-\sigma}})$.

On the other hand, our Theorem \ref{Thm1:LpPPP} improves Theorem \ref{thm:MRS} in two aspects. First of all, we require less regularity on the functions $u$ and $\lambda$. Indeed, we only need $L\loc^p$ solutions with $L^\infty\loc$ potentials in order to use the Kato inequality and the regularity result claimed in Section \ref{Sec:PreliminaryResults}. Secondly, we only need an $L^p$-bound on the asymptotic growth of $u^-$, instead of a pointwise asymptotic control. This allows us to consider a wider class of functions, for example having a super-quadratic growth, even in the case $\e<2$.

\end{remark}

In the particular case where $\e =2$, for instance when $\lambda$ is a constant, we get the next version of Theorem \ref{Thm1:LpPPP}.

\begin{corollary}\label{Cor0:LpPPP}
Let $(M,g)$ be a complete Riemannian manifold, $\lambda\in \lpl{\infty}{M}$ so that $\lambda \geq C$ for a positive constant $C$ and $p\in (1,+\infty)$.

If $u\in \lpl{p}{M}$ satisfies $-\Delta u + \lambda u \geq 0$ in the sense of distributions and
\begin{align}\label{Eq1:Cor0:LpPPP}
\int_{B_R} (u^-)^p \dvol = o(e^{\theta R}) \quad as\ R\to +\infty
\end{align}
with $\theta=\sqrt{\frac{(p-1)C}{e-1}}$, then $u\geq 0$ in $M$.
\end{corollary}

\begin{remark}
 Corollary \ref{Cor0:LpPPP} improves very much one of the main results of \cite{pigola2023approximation} in the setting of complete manifolds. Indeed, in that paper, the $L^{p}_{loc}$ positivity preservation is obtained under the condition $\int_{B_{R}(o)}{(u^{-})}^{p} \dvol = o(R^{2})$. See \cite[Corollary 5.2 and Remark 5.3]{pigola2023approximation}.
\end{remark}

As a byproduct, by applying Corollary \ref{Cor0:LpPPP} to both the functions $u$ and $-u$, we get an uniqueness statement for $L^p\loc$ solutions to $-\Delta u+\lambda u=0$.

\begin{corollary}[Uniqueness]\label{Cor2:Uniqueness}
Let $(M,g)$ be a complete Riemannian manifold, $\lambda\in \lpl{\infty}{M}$ so that $\lambda\geq C$ for a positive constant $C$ and $p\in (1,+\infty)$.

If $u\in \lpl{p}{M}$ satisfies $-\Delta u+\lambda u = 0$ in the sense of distributions and
\begin{align*}
\int_{B_R} (u^\pm)^p \dvol = o(e^{\theta R}) \quad as\ R\to +\infty
\end{align*}
with $\theta=\sqrt{\frac{(p-1)\lambda}{e-1}}$, then $u=0$ almost everywhere in $M$.
\end{corollary}

\begin{remark}
As already observed at the beginning of this section, for every $\lambda>0$ the function $u(x)=-e^{\sqrt{\lambda}x}$ provides a counterexample to the $\lpl{p}{\rr}$ positivity preserving property, for any $p\in (1,+\infty)$. Moreover, we stress that its $p$-norm has the following asymptotic growth
\begin{align*}
\int_{-R}^R (u^-)^p(x) \dint{x}=O(e^{p\sqrt{\lambda}R})
\end{align*}
with $p\sqrt{\lambda}>\sqrt{\frac{(p-1)\lambda}{e-1}}$. Therefore, Theorem \ref{Thm1:LpPPP} and Corollary \ref{Cor0:LpPPP} are not far from being sharp. It would be very interesting to understand to what extent this exponent can be refined.
\end{remark}


\section{$L^1\loc$ positivity preserving property}
\label{Sec:L1loc}

The approach used in Section \ref{Sec:Lploc}, which is based on inequality \eqref{Eq1:Regularity}, is clearly not applicable for $p=1$. To overcome this problem, we resort to some special cut-off to be used as test functions in the distributional inequality satisfied by $u$. The existence of these functions is ensured, for instance, by requiring certain conditions on the decay of the Ricci curvature.

\subsection{Cut-off functions with decaying laplacians}

The first theorem we present in this section is based on the following iterative lemma. It is an analogue of the Lemma \ref{Lem4:Inequality4} for the case $p=1$.

\begin{lemma}\label{Lem3:Inequality3}
Let $A>0$ and $f:[A,+\infty)\to (0,+\infty)$ be a nondecreasing function. Suppose there exist $\sigma>1, \gamma>0, \alpha>0$ and $\beta\geq 1$ so that
\begin{align}\label{Eq1:Inequality3}
f(r)\leq \frac{1}{\alpha r^\gamma+\beta} f(\sigma r)
\end{align}
for every $r\geq A$. Then, $f$ satisfies
\begin{align*}
f(R)\geq \left(\frac{R}{A}\right)^{\log_\sigma(\alpha A^\gamma +\beta)} \frac{f(A)}{\alpha A^\gamma+\beta}
\end{align*}
for every $R>A$.
\end{lemma}
\begin{proof}
Having fixed $R\geq A$, we have
\begin{align*}
f(R)&\leq (\alpha R^\gamma +\beta)^{-1} f(\sigma R)\\
& \leq (\alpha R^\gamma +\beta)^{-1} (\alpha (\sigma R)^\gamma +\beta)^{-1} f(\sigma^2 R)\\
&\leq (\alpha R^\gamma +\beta)^{-2} f(\sigma^2 R)
\end{align*}
and, iterating,
\begin{align*}
f(R)\leq (\alpha R^\gamma +\beta)^{-n} f(\sigma^n R)
\end{align*}
for every $n\in \mathbb{N}$.

Now consider $n\in \mathbb{N}$ so that $\sigma^{n+1} A\geq R \geq \sigma^n A$. In particular, from
\begin{align*}
\sigma^{n+1} A\geq R \quad \Rightarrow \quad n \geq \log_\sigma \left(\frac{R}{A}\right)-1
\end{align*}
we deduce
\begin{align*}
f(R)&\geq f(\sigma^n A) \geq (\alpha A^\gamma +\beta)^n f(A)\\
& \geq (\alpha A^\gamma +\beta)^{\log_\sigma\left(\frac{R}{A}\right)} \frac{f(A)}{\alpha A^\gamma +\beta}\\
& = \left(\frac{R}{A}\right)^{\log_\sigma(\alpha A^\gamma+\beta)} \frac{f(A)}{\alpha A^\gamma +\beta}
\end{align*}
as claimed
\end{proof}

As a consequence, by requiring the existence of a family $\{\phi_R\}_R$ of cut-off functions whose laplacians decay as $|\Delta \phi_R|\leq CR^{-\gamma}$ for a positive constant $\gamma$, we get

\begin{theorem}[Generalized $L^1\loc$ positivity preserving property]\label{Thm2:L1PPP}
Let $(M,g)$ be a complete Riemannian manifold and $\lambda$ a positive constant. Assume that for a fixed $o\in M$ there exist some positive constants $\gamma$ and $R_0$ and a constant $\sigma> 1$ satisfying the following condition: for every $R>R_0$ there exists $\phi_R\in \ccomp{2}{M}$ such that 
\begin{align}\label{Eq0:L1PPP}
\system{ll}{0\leq \phi_R \leq 1 & \inn M\\ \phi_R\equiv 1 & \inn B_R(o) \\ \textnormal{supp}( \phi_R) \subset B_{\sigma R}(o) \\ |\Delta \phi_R|\leq \frac{C}{R^\gamma} & \inn M}
\end{align}
where $C=C(\sigma)>0$ is a constant not depending on $R$.
If $u\in \lpl{1}{M}$ satisfies $-\Delta u+\lambda u \geq 0$ in the sense of distributions and there exists $k\in \mathbb{N}$ so that
\begin{align}\label{Eq1:L1PPP}
\int_{B_R(o)} u^- \dvol= O(R^k) \quad as\ R\to +\infty,
\end{align}
then $u\geq 0$ almost everywhere in $M$.
\end{theorem}
\begin{proof}
Fix $u\in \lpl{1}{M}$ a distributional solution to $-\Delta u+ \lambda u \geq 0$ that satisfies condition \eqref{Eq1:L1PPP} for a certain $k\in \mathbb{N}$. By Brezis-Kato inequality $\Delta u^- \geq \lambda u^-$ in the sense of distributions, implying
\begin{align*}
\lambda \int_M u^- \phi_R \dvol \leq \int_M u^- \Delta \phi_R \dvol \quad \quad \forall R>R_0.
\end{align*}
Using the definition of $\phi_R$, we get
\begin{align*}
\lambda \int_{B_R(o)} u^- \phi_R \dvol \leq \frac{C}{R^\gamma} \int_{B_{\sigma R}(o)\setminus B_R(o)} u^- \dvol \quad \quad \forall R>R_0
\end{align*}
and, by summing up
\begin{align*}
\frac{C}{R^\gamma} \int_{B_R(o)} u^- \dvol
\end{align*}
to both sides of the previous inequality, we obtain
\begin{equation}\label{App1:L1PPP}
\begin{split}
\int_{B_R(o)} u^- \dvol &\leq \frac{C}{\lambda R^\gamma +C} \int_{B_{\sigma R}(o)} u^- \dvol\\
&= \frac{1}{\alpha R^\gamma +1} \int_{B_{\sigma R}(o)} u^- \dvol \quad \quad \forall R>R_0,
\end{split}
\end{equation}
where $\alpha=\frac{\lambda}{C}$ depends on $\sigma$. \quad Similarly to what we done in Theorem \ref{Thm1:LpPPP}, if we suppose that $u^-\neq 0$ almost everywhere in $M$, then there exists $A\geq R_0$ so that
\begin{align*}
\int_{B_A(o)} u^- \dvol >0.
\end{align*}
By \eqref{App1:L1PPP} we can apply Lemma \ref{Lem3:Inequality3} to the function $f:[A,+\infty) \to \rr_{>0}$ given by
\begin{align*}
f:r\mapsto \int_{B_r(o)} u^- \dvol
\end{align*}
with $\beta=1$, and we get
\begin{align*}
f(R)\geq \left(\frac{R}{A} \right)^{\log_\sigma (\alpha A^\gamma+1)} \frac{f(A)}{\alpha A^\gamma + 1}
\end{align*}
for every $R>A$. Choosing $A\geq R_0$ big enough so that
\begin{align*}
\log_\sigma(\alpha A^\gamma +1) \geq k+1
\end{align*}
we have
\begin{align*}
f(R) \geq \left(\frac{R}{A} \right)^{k+1} \frac{f(A)}{\alpha A^\gamma + 1}
\end{align*}
for every $R>A$, thus obtaining a contradiction to \eqref{Eq1:L1PPP}. Hence $u^-=0$ almost everywhere, implying the claim.
\end{proof}

As showed by D. Bianchi and A.G. Setti in \cite[Corollary 2.3]{bianchi2018laplacian}, a sufficient condition for the existence of a family $\{\phi_R\}_R$ satisfying \eqref{Eq0:L1PPP} is a sub-quadratic decay of the Ricci curvature. Whence, we get the following corollary.

\begin{corollary}\label{Cor3:L1BianchiSetti}
Let $(M,g)$ be a complete Riemannian manifold of dimension $m$ and $\lambda$ a positive constant. Consider $o\in M$ and assume that 
\begin{align*}
\ric_g \geq -( m-1) C^2 (1+r^2)^\eta,
\end{align*}
where $C$ is a positive constant, $\eta \in [-1,1)$ and $r(x):=d(x,o)$ is the intrinsic distance from $o$ in $M$.
If $u\in \lpl{1}{M}$ satisfies $-\Delta u+\lambda u\geq 0$ in the sense of distributions and, for some $k \in \mathbb{N}$,
\begin{align*}
\int_{B_R(o)} u^- \dvol = O(R^k) \quad as\ R\to +\infty,
\end{align*}
then $u\geq 0$ almost everywhere in $M$.
\end{corollary}

\subsection{Cut-off functions with equibounded laplacians}

The second theorem of this section is an $L^{1}_{loc}$ positivity preserving property based on the existence of a family of cut-off functions with equibounded laplacians. The structure of the proof is very similar to the one adopted for Theorem \ref{Thm2:L1PPP} and it makes use of the following iterative lemma.

\begin{lemma}\label{Lem2:Inequality2}
Let $A>0$ and $f:[A,+\infty)\to (0,+\infty)$ be a nondecreasing function. Suppose there exist $\alpha>1$ and $\sigma>1$ so that
\begin{align}\label{Eq1:Inequality2}
f(r)\leq \frac{1}{\alpha} f(\sigma r)
\end{align}
for every $r\geq A$. Then, $f$ satisfies
\begin{align*}
f(R)\geq f(A) \left(\frac{R}{A \sigma} \right)^\theta
\end{align*}
for every $R>A$, where $\theta=\frac{\ln(\alpha)}{\ln(\sigma)}>0$.
\end{lemma}

\begin{proof}
Iterating \eqref{Eq1:Inequality2}, for every $n\in \mathbb{N}$ we get
\begin{align*}
f(r)\leq \frac{1}{\alpha^n} f(\sigma^n r)
\end{align*}
for any $r\geq A$. It follows that for any $R>A$
\begin{align*}
f(R)\geq f(A \sigma^n)\geq \alpha^n f(A) \geq \alpha^{\log_{\sigma}\left(\frac{R}{A\sigma}\right)} f(A)=f(A)\left(\frac{R}{A\sigma}\right)^{\frac{\ln(\alpha)}{\ln(\sigma)}},
\end{align*}
where $n=n(R,A,\sigma)$ is the unique natural number satisfying $\sigma^{n+1}\geq \frac{R}{A}\geq \sigma^n$. This concludes the proof.
\end{proof}

We can now state  our second main theorem that involves functions with an $L^{1}$-controlled growth.

\begin{theorem}[Generalized $L^1\loc$ positivity preserving property]\label{Thm3:L1PPP_V2}
Let $(M,g)$ be a complete Riemannian manifold and $\lambda$ a positive constant. Assume that for a fixed $o\in M$ there exist some positive constants $C$ and $R_0$ and a constant $\sigma>1$ satisfying the following condition: for every $R>R_0$ there exists $\phi_R\in \ccomp{2}{M}$ such that 
\begin{align}\label{Eq0:L1PPP_V2}
\system{ll}{0\leq \phi_R \leq 1 & \inn M\\ \phi_R\equiv 1 & \inn B_R(o) \\ \textnormal{supp}( \phi_R) \subset B_{\sigma R}(o) \\ |\Delta \phi_R|\leq C & \inn M.}
\end{align}
If $u\in \lpl{1}{M}$ satisfies $-\Delta u+\lambda u \geq 0$ in the sense of distributions and
\begin{align}\label{Eq1:L1PPP_V2}
\int_{B_R(o)} u^- \dvol = o(R^\theta) \quad as\ R\to +\infty
\end{align}
with $\theta=\frac{\ln\left(1+\frac{\lambda}{C}\right)}{\ln(\sigma)}$, then $u\geq 0$ almost everywhere in $M$.
\end{theorem}
\begin{proof}
As in the proof of Theorem \ref{Thm2:L1PPP}, we get
\begin{align*}
\lambda \int_M u^- \phi_R \dvol \leq \int_M u^- \Delta \phi_R \quad \quad \forall R>R_0
\end{align*}
that implies
\begin{align}\label{App1:L1PPP_V2}
\int_{B_R(o)} u^- \dvol \leq \frac{C}{\lambda +C} \int_{B_{\sigma R}(o)} u^- \dvol \quad \quad \forall R>R_0.
\end{align}
If $u^-\neq 0$ almost everywhere in $M$, then there exists $A>0$ such that
\begin{align*}
\int_{B_A(o)}u^- \dvol >0.
\end{align*}
By \eqref{App1:L1PPP_V2} we can apply Lemma \ref{Lem2:Inequality2} with
\begin{align*}
f:r\mapsto \int_{B_r(o)}u^- \dvol
\end{align*}
and $\alpha=\frac{C+\lambda}{C}$ and we deduce that $f(R)\geq C_0 R^{\frac{\ln(\alpha)}{\ln(\sigma)}}$ for every $R>A$, where $C_0>0$. This contradicts \eqref{Eq1:L1PPP_V2}. Hence $u^-=0$ almost everywhere in $M$, as required.
\end{proof}

In the proof of \cite[Corollary 4.1]{impera2022higher}, the authors obtained a family of cutoff functions satisfying \eqref{Eq0:L1PPP_V2} under the only assumption of a lower bound on the Ricci curvature. As a consequence, we obtain the following
\begin{corollary}\label{Cor4:L1MariniVeronelli}
Let $(M,g)$ be a complete Riemannian manifold and $\lambda$ a positive constant. Consider $o\in M$ and assume that 
\begin{align*}
\ric_g(x) \geq -G^2 (r(x))
\end{align*}
for every $x\in M\setminus B_R(o)$, where $G\in C^\infty$ is given by
\begin{align*}
G(t)=\alpha t \prod_{0\leq j\leq k} \ln^{[j]}(t)
\end{align*}
for $t>1$, $\alpha>0$ and $k\in \mathbb{N}$.
Then, there exists a constant $\theta=\theta(\lambda,M,\alpha, k)>0$ such that if $u\in \lpl{1}{M}$ satisfies $-\Delta u+\lambda u\geq 0$ in the sense of distributions and
\begin{align*}
\int_{B_R(o)} u^- \dvol = o(R^\theta) \quad as\ R\to +\infty,
\end{align*}
then $u\geq 0$ almost everywhere in $M$. In particular, the $L^1$ positivity preserving property holds true in the family of functions
\begin{align*}
\{u\in \lpl{1}{M}\ :\ u^- \in L^1(M)\}.
\end{align*}
\end{corollary}

\begin{remark}
In \cite[Theorem B]{bisterzo2022infty}, the authors constructed a counterexample to the $L^1$ positivity preserving property in a complete 2-manifold having Gaussian curvature with an asymptotic of the form $K(x)\sim -C r(x)^{2+\e}$, for $\e>0$. This underline that the result contained in Corollary \ref{Cor4:L1MariniVeronelli} is sharp.
\end{remark}

\begin{remark}
When stated in terms of a Liouville type property, our Corollary  \ref{Cor4:L1MariniVeronelli} compares e.g. with \cite[Theorem C]{rigoli2001liouville}, where the authors consider the case $\lambda= 0$ of subharmonic functions. Their result states that a $C^{1}$, nonnegative subharmonic function with precise pointwises exponential control and a logarithmic $L^{1}$ growth must be constant. They also provide a rotationally symmetric example $(M,g)$ with Gaussian curvature $K(x) \sim -C r(x)^{2}$ showing that, without the pointwise control, there exists an unbounded smooth solution of $\Delta u = 1$ of logarithmic $L^{1}$-growth. As a consequence, keeping the curvature restriction of Corollary \ref{Cor4:L1MariniVeronelli}, in order to obtain the Liouville result under a pure   $L^{1}$-growth condition, which is even faster than logarithmic, one has to assume that $\lambda >0$.
\end{remark}


\section{An application to minimal submanifolds}\label{Sec:Applications}

Recall that an immersed submanifold  $x:\Sigma^{n} \hookrightarrow \rr^{m}$ is said to be minimal if its mean curvature vector field satisfies $H^\Sigma=0$. It is a standard fact that the minimality condition is equivalent to the property that the coordinate functions of the isometric immersion are harmonic, i.e.,
\begin{align*}
\Delta^{\Sigma}\, x_i=0 \quad \forall i=1,...,m.
\end{align*}
In particular, this implies that for any minimal submanifold in Euclidean space,
\begin{align*}
\Delta^{\Sigma}\,  |x|^2 = 2n.
\end{align*}

As an application of the main results in Section \ref{Sec:Lploc}, we prove that complete minimal submanifold enjoy the following $L^{p}$ extrinsic distance growth condition.

\begin{corollary}\label{Thm:MinimalHypersurfaces}
Let $x: \Sigma\hookrightarrow \rr^m$ be a complete minimal submanifold and suppose there exists a positive function $\xi: \rr_{\geq 0} \to \rr_{> 0}$ such that
\begin{align*}
(d^{\rr^{m}}(x,o))^2\leq \xi(d^\Sigma(x,o))  \qquad \text{and}  \qquad  \xi(R) = O(R^{2-\e}),\, \text{as }R\to+\infty 
\end{align*}
for some constants $C>0$ and $\e\in (0,2]$ and for some fixed origin $o\in \Sigma$. Then, for every $p\in (1,+\infty)$,
\begin{align}\label{Eq:MinimalLimsup}
\limsup_{R\to +\infty} \frac{\int_{B_R^\Sigma(o)} \xi^{p} \dvol_{\Sigma}}{e^{\theta R^\frac{\e}{2}}}>0,
\end{align}
where $\theta=\sqrt{\frac{(p-1)C}{e-1}}$.
\end{corollary}

\begin{proof}
Without loss of generality we can suppose $o=0\in \rr^{m}$. Let
\[
w(x):=d^{\rr^{m}}(x,o)=|x|^2
\]
and define
\[
\lambda(x):=\frac{2n}{\xi(d^{\Sigma}(x,o))}.
\]
Then
\begin{align*}
\Delta^\Sigma\,  w=2n=\lambda \xi \geq \lambda w.
\end{align*}
By contradiction, suppose that  \eqref{Eq:MinimalLimsup} is not satisfied for some $p\in (1,+\infty)$. Then
\begin{align*}
0 = \limsup_{R\to +\infty} \frac{\int_{B_R^\Sigma(o)} \xi^p \dvol_{\Sigma}}{e^{\theta R^\frac{\e}{2}}} \geq \limsup_{R\to +\infty} \frac{\int_{B_R^\Sigma(o)} w^p \dvol_{\Sigma}}{e^{\theta R^\frac{\e}{2}}} \geq 0,
\end{align*}
showing that
\[
\int_{B^{\Sigma}_{R}(o)}w^{p} \dvol_{\Sigma} = o(e^{\theta R^\frac{\e}{2}}),\, R \to +\infty.
\]
An application of Theorem \ref{Thm1:LpPPP}, in the form of a Liouville type result, yields that $w \equiv 0$. Contradiction.
\end{proof}

\begin{remark}
In the assumption of Corollary \ref{Thm:MinimalHypersurfaces} we get an asymptotic estimate on the behavior of $|B_R^\Sigma|$. Indeed, since there exist two constants $C>0$ and $\e \in (0,2]$ such that
\begin{align*}
\xi(x)\leq C \left(1+d^\Sigma(x,o) \right)^{2-\e},
\end{align*}
then
\begin{align*}
\int_{B_R^\Sigma(o)} \xi^p \dvol \leq  C \int_{B_R^\Sigma(o)} (1+d^\Sigma(x,o))^{(2-\e)p} \dvol \leq C (1+R)^{(2-\e)p} |B_R^\Sigma(o)|.
\end{align*}
By \eqref{Eq:MinimalLimsup} it follows
\begin{align*}
\limsup_{R\to +\infty} \frac{(1+R)^{(2-\e)p} |B_R^\Sigma(o)|}{e^{\theta R^\frac{\e}{2}}}>0.
\end{align*}
Whence, we obtain the validity of the following nonexistence result.
\end{remark}

\begin{corollary}
There are no complete minimal submanifolds $\Sigma^n\hookrightarrow \rr^{m}$ satisfying the following conditions:
\begin{enumerate}
\item [a)] the extrinsic distance from a fixed origin $o \in \Sigma$ satisfies
\begin{align*}
\left(d^{\rr^{m}}(x,o)\right)^2\leq \xi(d^\Sigma (x,o))
\end{align*}
with 
\begin{align*}
\xi(R)=O(R^{2-\e}) \quad \text{as } R \to +\infty
\end{align*}
for some $\e \in (0,2]$;

\item [b)] the intrinsic geodesic balls of $\Sigma$ centered at $o$ satisfy the asymptotic estimate
\[
|B_R^\Sigma(o)|=o\left(R^{-(2-\e)p}e^{\theta R^{\frac{\e}{2}}} \right)\, \text{ as } R\to +\infty,
\]
with $\theta=\sqrt{\frac{(p-1)C}{e-1}}$ and $p\in(1,+\infty)$.
\end{enumerate}
\end{corollary}

\begin{remark}
We stress that in case $\e=2$, i.e. for bounded minimal submanifolds, the volume growth we obtained is far from being optimal. Indeed, in \cite{Karp-MathAnn} and \cite{PRS-PAMS} the authors achieved the rate $|B_R^\Sigma(o)|=O(e^{CR^2})$. This discrepancy comes from the fact that we use integral techniques and estimates.
\end{remark}

\bibliography{references}
\bibliographystyle{abbrv}

\end{document}